\documentclass[12pt,reqno]{amsart}
\usepackage{amssymb,enumerate}

\usepackage[curve,matrix,arrow]{xy}

\usepackage{pdfsync}

\baselineskip

\theoremstyle{plain} 
\numberwithin{equation}{section}
\newtheorem{thm}[equation]{Theorem}

\newtheorem{cor}[equation]{Corollary}

\theoremstyle{definition}
\newtheorem{defi}[equation]{Definition}
\newtheorem{rem}[equation]{Remark}
\newtheorem{exm}[equation]{Example}
\newtheorem{construction}[equation]{Construction}

\def\bP{{\mathbb P}}

\newcommand{\gDelta}{\Delta_{\mathrm{Gr}}}
\newcommand{\grS}{\sigma_{\mathrm{Gr}}}
\newcommand{\lDelta}{\Delta_{\mathrm{Lie}}}
\newcommand{\LiS}{\sigma_{\mathrm{Lie}}}
\newcommand{\sDelta}{{\scriptstyle{\Delta}}}

\newcommand{\env}[1]{{#1}^{\operatorname{e}}}
\def\Ext{\operatorname{Ext}\nolimits}

\def\CE{{\mathcal{E}}}

\def\bfp{{\mathbb{F}_p}}
\def\HH{\operatorname{H}\nolimits}
\def\HcH{\operatorname{HH}\nolimits}

\newcommand{\op}[1]{{#1}^{\operatorname{op}}}

\def\proj{\operatorname{(proj)}\nolimits}

\newcommand{\sst}{\scriptscriptstyle}

\def\xra{\xrightarrow}

\def\lra{\longrightarrow}

\title[Hopf algebras and tensor products]{Hopf algebra structures and
  \\ tensor products for group algebras} 
\author[J.~F.~Carlson]{Jon F. Carlson}
\address{Department of Mathematics,
University of Georgia, Athens, GA 30602, USA}
\email{jfc@math.uga.edu}

\author[S.~B.~Iyengar]{Srikanth B. Iyengar}
\address{Department of Mathematics, University of Utah, 
Salt Lake City, UT 68588, USA}
\email{iyengar@math.utah.edu}

\thanks{JFC was partially supported by NSA  grant H98230-15-1-0007
  and by Simons Foundation grant 054813-01. SBI was partially
  supported by NSF grant DMS-1503044. JFC would like to thank the
  University of Utah for kind hospitality during his visit when
  the work on this paper was started.}

\date\today

\begin{document}

\begin{abstract} 
  The modular group algebra of an elementary abelian $p$-group is
  isomorphic to the restricted enveloping algebra of commutative
  restricted Lie algebra. The different ways of regarding this
  algebra result in different Hopf algebra structures that
  determine cup products on cohomology of modules. However,
  it is proved in this paper that the products with  elements
  of the polynomial subring of the cohomology ring generated
  by the Bocksteins of the degree one elements are independent
  of the choice of these coalgebra structures. 
\end{abstract}

\maketitle

\section{Introduction}

This paper concerns the group algebra $kE$ of an elementary
abelian $p$-group $E$ of order $p^r$ over a field $k$ of
characteristic $p$. This algebra has a natural coalgebra
structure $kE \to kE \otimes_{k} kE$ given by $g \mapsto g \otimes g$
for each $g$ in $E$. On the other hand, if
$E= \langle g_1, \dots, g_r \rangle$, a change of variables
$x_i = g_i -1$, realizes $kE$ as a truncated polynomial ring
$k[x_1, \dots, x_r]/(x_1^p, \dots, x_r^p)$. This is isomorphic
to the restricted enveloping algebra of the restricted
$p$-Lie algebra $k^{r}$ with trivial bracket and $p$-power
operation. Again, there is a natural Hopf algebra structure,
this time given by the map $x_i \mapsto x_{i} \otimes 1 + 1 \otimes x_i$.
The two coalgebra structures are not the same and they define
different tensor products on $kE$-modules as well as different
actions of the cohomology ring
$H^{*}(E,k)\cong \Ext^{*}_{kE}(k,k)$ on $\Ext^*_{kE}(M,N)$
for $kE$-modules $M,N$.

The differences in the Hopf structure has shown up in several
works. For example, Avrunin and Scott~\cite{AS} exploited a change
in the coalgebra structure to prove a conjecture of the first
author~\cite{Cvar} that the rank variety and the support variety
of a $kE$-module are homeomorphic. In \cite{CFP,FP} the authors
define bundles on projective space using modules of constant
Jordan type and the Lie coalgebra map. The construction is not
available with the group coalgebra map. Both of these works
used the fact that with the Lie algebra structure there is an
abundance of sub-Hopf algebras generated by units in the
algebra. The immediate motivation for this paper is the desire
to make efficient use of categorical equivalences and  functors
relating commutative algebra and group representation theory;
see \cite{CI}.  The fact that the Hopf algebra structures
differ has been an obstruction to this end.

For any Hopf algebra $A$ over $k$ and $A$-module $M$, the
cohomology ring $\Ext^*_{A}(M,M)$ is a module over the cohomology
ring $\Ext^*_{A}(k,k)$. The action is given by a homomorphism of rings
\[
\theta^{M}\colon \Ext^*_{A}(k,k) \lra \Ext^*_{A}(M,M)
\]
that can be described as follow: take a homogeneous element
$\zeta$ of $\Ext^{n}_{A}(k,k)$, regard is as a length-$n$
exact sequence beginning and ending in the trivial module $k$,
then tensor over $k$ with $M$. The image is the class of that
sequence. The map, in general, depends on the coalgebra structure.
The primary result of this paper is that for the group algebra
of an elementary abelian group the dependence is not so bad.

Specifically, for $E$ an elementary abelian $p$-group, if $S$ is
the polynomial subring of $\HH^*(E,k)$ generated by the Bocksteins
of the degree one elements, then the restriction of $\theta^{M}$
to $S$ is the same for both the group and the Lie coalgebras
structures on $kE$. As a direct corollary one gets that for
$\zeta \in S$ and $L_{\zeta}$ the $kE$-module introduced in \cite{C2},
the isomorphism class of  $L_\zeta \otimes_{k} M$  does not depend
on which of the two Hopf algebra structures is used to define
the action on the tensor product.

A key input in our work is the fact, proved by Pevtsova and
Witherspoon~\cite{PW}, that for any Hopf algebra $A$, the map
$\theta^{M}$ factors through the Hochschild cohomology ring
$\HcH^{*}(A/k;A)$. The advantage gained by this observation
is that the first map, to $\HcH^{*}(A/k;A)$, depends on the
coalgebra structure and not on $M$, while the second depends on $M$
and not on the choice of coalgebra structures. So it is
sufficient to show that, for $A=kE$, the first map is the
same on the elements of $S$ regardless of the coalgebra.
This is accomplished by a straightforward calculation using
the fact that $E$ is a direct product of cyclic groups.

Section \ref{sec:Hopf} of the paper is devoted to preliminaries
on Hopf algebras and cohomology, mainly a detailed proof of
the factorization of $\theta^{M}$ discussed above.  Basic
facts about the cohomology of elementary abelian $p$-groups
are recalled in Section \ref{sec:prelim}, while Section
\ref{sec:eleab} presents a proof of the main theorem.
Results on the tensor products of $L_\zeta$ modules are
presented in Section \ref{sec:lzeta}.

\section{Hopf algebras and cohomology}
\label{sec:Hopf}
This section concerns the cohomology of modules over Hopf
algebras.  The main result is Theorem~\ref{thm:modulecoho},
due to Pevtsova and Witherspoon~\cite[Lemma~13]{PW}.
We present a detailed proof because the constructions
of the maps involved in the statement of the result  are
of  critical importance in the next section. 

Let $k$ denote a field and $A$ a Hopf algebra over $k$,
with unit $\varepsilon\colon A \to k$,  coalgebra  map
$\Delta\colon A \to A \otimes_k A$, and counit
$\eta\colon k\to A$. We assume  that $A$ has an
antipode $\sigma$, that is to say, $\sigma$ is the
inverse of 
the identity on $A$, under the convolution product. We
adapt Sweedler's notation and write
\[
\Delta(\alpha) = \sum_{(\alpha)}\alpha_{1}\otimes\alpha_{2}
\quad\text{for $\alpha\in A$.}
\]
Unless stated to the contrary, the term ``module" is assumed
to mean a finitely generated left module. 

\begin{construction}
\label{con:ktoM}
Let $M$ be an $A$-module. Recall that for each $A$-module $X$,
there is a structure of an $A$-module 
on $X\otimes_{k}M$ induced by the diagonal:
\[
\alpha \cdot (x\otimes m) = \sum_{(\alpha)}\alpha_{1}x\otimes \alpha_{2}m
\]
The assignment $X\mapsto X\otimes_{k}M$ defines an additive
functor, that we denote $\theta^{M}_{\Delta}$, on the category
of $A$-modules, and has the following properties.
\begin{enumerate}[\quad\rm (1)]
\item
  The natural map $M\to k\otimes_{k}M=\theta^{M}_{\Delta}(k)$
  that sends $m$ to $1\otimes m$ is an isomorphism of left
  $A$-modules.
\item
  When $M$ is projective, so is the $A$-module
  $\theta^{M}_{\Delta}(X) = X \otimes_{k} M$.
\end{enumerate}
These are standard computations.  It follows that there
is an induced homomorphism of graded $k$-algebras:
\begin{equation}
\label{eq:ktoM}
\Theta_{\!\sDelta}^{\sst M} \colon \Ext_{A}^{*}(k,k)\lra 
\Ext_{A}^{*}(M,M)\,.
\end{equation}
The notation is intended to emphasize the fact that the map
depends on the coalgebra structure on $A$.
\end{construction}

We write $\env A$ for the enveloping algebra
$A\otimes_{k}\op A$ of $A$. Since $k$ is a field, the
Hochschild cohomology of $A$ as a $k$-algebra can be
introduced as
\[
\HcH^{*}(A/k;A) = \Ext_{\env A}^{*}(A,A)\,.
\]
An $\env A$-module is the same thing as a left-right
$A$-bimodule. In particular, $A$ is naturally an
$A^e$-module, with action defined by
$(\alpha\otimes\beta)\cdot a= \alpha a\beta$.

\begin{construction}
\label{con:AtoM}
Given an $A$-module $M$ and an $\env A$-module $Y$, there
is a residual $A$-module structure on $Y\otimes_{A}M$, defined by 
\[
\alpha \cdot (y\otimes m) = (\alpha y)\otimes m\,,
\]
for $\alpha \in A$, $y \in Y$ and $m \in M$. The
assignment $Y\mapsto Y\otimes_{A}M$ is an additive
functor, denoted $\psi^{M}$, from $\env A$-modules
to $A$-modules.  The next assertions are immediate.
\begin{enumerate}[\quad\rm(1)]
\item
  The natural map $\psi^{M}(A)= A\otimes_{A}M\to M$ that
  sends $a\otimes m$ to $am$, is an isomorphism of $A$-modules.
\item
  When $P$ is a projective $\env A$-module, the $A$-module
  $\psi^{M}(P)$ is projective.
\end{enumerate}
It follows that $\psi^{M}$ induces a  homomorphism of
graded $k$-algebras:
\begin{equation}
\label{eq:AtoM}
\Psi^{M}\colon \HcH^{*}(A/k;A)\lra \Ext^{*}_{A}(M,M)\,.
\end{equation}
Note that this map is entirely independent of the coalgebra
structure on $A$.
\end{construction}

\begin{construction}
\label{con:ktoA}
Let $X$ be an $A$-module. Then $X\otimes_{k}A$ has a structure
of an $A$-module induced by the diagonal $\Delta$. It also
has a right 
$A$-module action induced by the right action of $A$ on itself.
In short, $X\otimes_{k}A$ is a left $\env A$-module, with
action determined by 
\[
(\alpha \otimes \beta) \cdot (x\otimes a) = 
\sum_{(\alpha)}\alpha_{1}x\otimes \alpha_{2}a\beta
\]
The assignment $X\mapsto X\otimes_{k}A$ defines an additive
functor, that we denote $\phi_{\Delta}$, from $A$-modules 
to  $\env A$-modules. This has the following properties.
\begin{enumerate}[\quad\rm (1)]
\item 
  The natural isomorphism $A\xra{\cong} k\otimes_{k}A= \phi_{\Delta}(k)$,
  mapping $a\to 1\otimes a$, is one of 
$\env A$-modules, where the $\env A$-action on $A$ is the usual one.
\item 
  The $\env A$-linear map
  $\iota\colon \env A\to \phi_{\Delta}(A)= A\otimes_{k}A$ where
  $1\otimes 1$ maps to $1\otimes 1$, is an isomorphism, with
  inverse defined by the assignment
\[
\alpha\otimes\beta\mapsto \sum_{(\alpha)} \alpha_{1}\otimes 
\sigma(\alpha_{2})\beta\,.
\]
In particular, $\phi_{\Delta}(A)$ is a free $\env A$-module of
rank one, and $\phi_{\Delta}(P)$ is projective whenever $P$ is a projective
$A$-module. 
\end{enumerate}

Statement (1) is readily verified, given that
$\varepsilon \colon A\to k$ is the counit of the coalgebra
structure on  $A$;  that is to say,
for any $\alpha\in A$, one has
\[
\sum_{(\alpha)}\varepsilon(\alpha_{1})\alpha_{2} = \alpha\,.
\]
As to (2), since the map $\iota$ is $\env A$-linear, by
construction, it suffices to verify that its composition
with the given map (henceforth denoted $\iota^{-1}$, in
anticipation) is the identity. Moreover, $\iota^{-1}$ is
evidently a homomorphism of right $A$-modules, and since 
\[
\iota(\alpha\otimes 1) = (\alpha\otimes 1)\cdot (1\otimes 1) =
\sum_{(\alpha)}\alpha_{1}\otimes \alpha_{2}
\]
it suffices to verify that $\iota^{-1}$ maps the term on
the right to $\alpha\otimes 1$, for any $\alpha\in A$.
To this end, recall that 
$\Delta$ is coassociative, so that  
\[
\sum_{(\alpha)} \sum_{(\alpha_{1})}\alpha_{11}\otimes
\alpha_{12} \otimes \alpha_{2} 
= \sum_{(\alpha)}\sum_{(\alpha_{2})}\alpha_{1}\otimes
\alpha_{21}\otimes \alpha_{22}
\]
This explains the second of the following equalities.
\begin{align*}
\iota^{-1}\big(\sum_{(\alpha)}\alpha_{1}\otimes \alpha_{2}\big) 
&= \sum_{(\alpha)}\sum_{(\alpha_{1})}\alpha_{11}\otimes
\sigma(\alpha_{12})\alpha_{2} \\
&= \sum_{(\alpha)}\sum_{(\alpha_{2})}\alpha_{1}\otimes
\sigma(\alpha_{21})\alpha_{22} \\
& = \sum_{(\alpha)} \alpha_{1}\otimes
\varepsilon(\alpha_{2})  \\
& = \big(\sum_{(\alpha)} \alpha_{1}
\varepsilon(\alpha_{2})\big) \otimes 1 \\
	 & = \alpha \otimes 1
\end{align*}
The third equality a consequence of the definition of the
anitpode and the last equality holds because $\varepsilon$
is the counit of the coalgebra structure on $A$. 

Given properties (1)  and (2) of $\phi_{\Delta}$, it is
immediate that it induces a homomorphism of graded $k$-algebras:
\begin{equation}
\label{eq:ktoA}
\Phi_{\Delta}\colon \Ext_{A}^{*}(k,k)\lra \HcH^{*}(A/k;A)\,.
\end{equation}
\end{construction}

The result below, proved by Pevtsova and Witherspoon~\cite{PW},
links the three homomorphisms, \eqref{eq:ktoM}, \eqref{eq:AtoM}, 
and \eqref{eq:ktoA}, constructed above. 

\begin{thm} 
\label{thm:modulecoho}
Let $A$ be a Hopf algebra over $k$.  For each $A$-module $M$,
the following diagram of graded $k$-algebras commutes.
\[
\xymatrix{
\Ext^{*}_{A}(k,k) \ar[drr]^{\Theta_{\!\sDelta}^{\sst M}} \ar[dd]_{\Phi_{\Delta}} \\
&& \Ext^{*}_{A}(M,M) \\
\HcH^*(A/k;A) \ar[urr]_{\Psi^{M}}
}
\]
\end{thm}

\begin{proof}
  Let $X$ be an $A$-module. Using the description of the
  $A$-action on $X\otimes_{k}M$ and the $\env A$-action
  on $\phi_{\Delta}(X)$, 
it is a direct verification that the canonical bijection 
\[
X\otimes_{k} M  \lra  (X\otimes_{k}A)\otimes_{A} M 
= \phi_{\Delta}(X)\otimes_{A}M
\quad\text{where $x\otimes m\mapsto  x\otimes 1 \otimes m$}
\]
is compatible with the $A$-module structures. It yields an
isomorphism of functors $\theta^{M}_{\Delta}\cong \psi^{M}\phi_{\Delta}$
on the category of $A$-modules. Since $\phi_{\Delta}$  take
projectives to projectives, it follows that there is an equality of 
induced functors. This is the stated result.
\end{proof}

\begin{defi}
\label{defi:Lzeta}
Let $k$ be a field and $A$ a $k$-algebra.  In what follows, we
say that $A$-modules $M$ and $N$ are \emph{stably isomorphic} 
if there exist projective $A$-modules $P$ and $Q$ such that
$M\oplus P \cong N\oplus Q$. 

Let $P_{*}$ be a projective resolution of an $A$-module $M$.
For any integer $d\ge 0$, the image of the boundary map
$\partial\colon P_d  \to P_{d-1}$ is independent of the choice
of $P$, up to a stable isomorphism. We denote it $\Omega^{d}(M)$,
and call it a $d$th syzygy module of $M$.

Fix an element $\zeta\in \Ext^{d}_{A}(k,k)$ and a $d$th
syzygy module $\Omega^{d}(k)$. Then $\zeta$ is represented
by a  homomorphism on $\Omega^d(k)$, that we also call $\zeta$.
So we get an exact sequence of $A$-modules:
\begin{equation}
\label{eq:Lzeta}
\xymatrix{
0 \ar[r] & L_\zeta \ar[r] & \Omega^{d}(k) \ar[r]^{ \ \ \zeta} 
& k \ar[r] & 0
}
\end{equation}
That is, the module $L_\zeta$ is defined to be the kernel  of
map $\zeta$ on $\Omega^{d}(k)$.  Up to a stable isomorphism,
this is independent of  the choice of a syzygy module.
\end{defi}

Given a $k$-algebra $A$, we say that a map
$\Delta\colon A\to A\otimes_{k}A$
\emph{induces a Hopf structure  on $A$} if there exists
a Hopf algebra structure on $A$  (and this includes an antipode)
with $\Delta$ as the comultiplication. For ease of comprehension,
given a coalgebra map $\Delta$ and $A$-modules $X,M,$ the
$A$-module defined on the vector space $X\otimes_{k}M$
using the Hopf structure $\Delta$
is denoted
\[
\Delta(X\otimes_{k}M)
\]
This is precisely the module $\theta^{M}_{\Delta}(X)$ defined
in  Construction~\ref{con:ktoM}. 

\begin{cor}
\label{cor:HH}
Let $\Delta_{1},\Delta_{2}\colon A\to A\otimes_{k}A$ be maps
that induce Hopf algebra structures on $A$.  If
$\zeta\in \Ext^{d}_{A}(k,k)$ is such that
$\Phi_{\Delta_{1}}(\zeta) = \Phi_{\Delta_{2}}(\zeta)$, then for
each $A$-module $M$,  the $A$-modules
$\Delta_{1}(L_{\zeta}\otimes_{k}M)$ and
$\Delta_{2}(L_{\zeta}\otimes_{k}M)$ are stably isomorphic.
\end{cor}

\begin{proof}
  Let $P_{*}$ be a projective resolution of $k$. For $i=1,2$
  the complex $\theta^{M}_{\Delta_{i}}(P_{*})$ is a projective
  resolution of $\theta^{M}_{\Delta_{i}}(k)\cong M$. Thus, the
  $A$-modules $\theta_{\Delta_{i}}^{M}(\Omega^{d}(k))$ and
  $\Omega^{d}(M)$  are stably isomorphic.  Therefore, the
  exact sequence \eqref{eq:Lzeta} induces an exact sequence
\[
\xymatrixcolsep{1.5pc}
\xymatrix{
0 \ar[r] & \theta_{\Delta_{i}}^{M}(L_\zeta) \ar[r] & \Omega^{d}(M) 
\ar[rr]^-{\Theta_{\Delta_{i}}^{M}(\zeta)} \oplus \proj && M \ar[r] & 0
}
\]
of $A$-modules, where $\Theta_{\Delta_{i}}^{M}$ is the
map~\eqref{eq:ktoM}. Since  $\Delta_{1}(\zeta) = \Delta_{2}(\zeta)$,
by hypothesis, it follows from Theorem~\ref{thm:modulecoho}
that $\Theta_{\Delta_{1}}^{M}(\zeta)=\Theta_{\Delta_{2}}^{M}(\zeta)$.
This yields the desired result.
\end{proof}

\begin{rem}
\label{rem:minimal}
Assume that the $k$-algebra $A$ is finite dimensional.
Then finitely generated modules over $A$ admit minimal
projective resolutions, and hence syzygy modules are
well-defined, up to isomorphism of $A$-modules. What is more,
each $\zeta\in\Ext^{d}_{A}(k,k)$ is represented by a \emph{unique}
homomorphism $\Omega^{d}(k)\to k$, and then setting $L_{\zeta}$
to be its kernel pins down the latter, up to isomorphism.
In the same vein, $\Omega^{d}(k)\otimes_{k}M \cong \Omega^{d}(M)$,
so we gets a well-defined module $L_{\zeta}\otimes_{k}M$. 

It then follows from the argument in Corollary~\ref{cor:HH}
that if $\Phi_{\Delta_{1}}(\zeta) = \Phi_{\Delta_{2}}(\zeta)$
the $A$-modules 
$\Delta_{1}(L_{\zeta}\otimes_{k}M)$ and
$\Delta_{2}(L_{\zeta}\otimes_{k}M)$ are in fact isomorphic.
\end{rem}

\section{Cohomology of elementary abelian $p$-groups} 
\label{sec:prelim}
In this section, we set  notation and review some facts about  the cohomology, and  Hochschild cohomology, of elementary abelian $p$-groups; see \cite[Section 4.5]{CTVZ} and \cite{Ho} for details. Throughout $k$ will be a field of positive characteristic $p$. 

Let $E:=\langle g \rangle$ be a cyclic group of order $p$.
Setting $x = g-1$ we may write $A:= kE$, the group algebra
of $E$ over $k$, as a truncated polynomial ring $A \cong k[x]/(x^p)$.
Consider the complex of projective $A$-modules
\begin{equation}
\label{eq:kres}
\xymatrix{
  P_{*}\colon & \cdots \ar[r] & A \ar[r]^x & A \ar[r]^{x^{p-1}}
  & A \ar[r]^x & A  \ar[r] & 0\,,
}
\end{equation}
that is nonzero in degrees $\ge 0$. The augmentation
$\varepsilon \colon P_{*} \to k$, that maps $P_{i}$ to zero
for $i>0$ and is the canonical surjection for $i=0$, is a
morphism of complexes, and a quasi-isomorphism; thus
$(P_{*},\varepsilon)$ is a minimal projective resolution
of $k$, over $A$.

Let $E := \langle g_1, \dots, g_r \rangle$ be an elementary
abelian group of order $p^r$. For each integer $i=1,\dots,r$,
set $A_i:=k[x_i]/(x_i^p)$. Then
$A := A_1 \otimes_k \dots \otimes_k A_r$ is the group algebra
of $E$, where $x_i = g_i-1$ for each $i$.  With
$(P^{(i)}_*, \varepsilon_i)$ the projective $A_i$-resolution
of $k$, from \eqref{eq:kres}, the complex
\begin{equation}
\label{eq:kres2}
(P_{*},\varepsilon):= (P^{(1)}_* \otimes_{k} \cdots \otimes_{k} P_*^{(r)},
\varepsilon_1 \otimes \cdots \otimes \varepsilon_r)\,.
\end{equation}
is a projective $A$-resolution of $k$.  Set 
\[
P_{j_1, \cdots, j_r} := P_{j_1}^{(1)} \otimes_{k} \cdots \otimes_{k}
P_{j_r}^{(r)} = A_1 \otimes_{k} \cdots \otimes_{k} A_r = A
\]
and let $\theta_{j_1, \dots, j_r}\colon P_* \to k$ be the map
whose restriction to $P_{j_1, \dots, j_r}$ is the augmentation
$A\to k$ and whose restriction to $P_{\ell_1, \dots, \ell_r}$
is zero if $j_i \ne \ell_i$ for some $i$. Let
$\hat{\eta_i}:= \theta_{j_1, \dots, j_r}$ where
$j_i = 1$ and $j_\ell = 0$ for $\ell \neq i$. Let
$\hat{\zeta_i}:=\theta_{j_1, \dots, j_r}$ where $j_i = 2$
and $j_\ell = 0$ for $\ell \neq i$.

The cohomology ring of $A$ has the form
\[
\Ext_{A}^{*}(k,k) = \begin{cases} 
k[\eta_1, \dots, \eta_r] & \text{ if $p = 2$, } \\
\Lambda(\eta_1, \dots, \eta_r) \otimes_{k}
k[\zeta_1, \dots, \zeta_r]  &\text{ otherwise,}
\end{cases}
\]
where each $\eta_i$ is represented by the cocycle
$\hat{\eta}_i$ and each $\zeta_i$ is represented by
$\hat{\zeta}_i$. Here $\eta_i$ is in degree $1$ and $\zeta_i$
is in degree $2$. Since the resolution $(P_*, \varepsilon)$
is minimal, $\eta_i$ is uniquely represented by  $\hat{\eta}_i$
and $\zeta_i$ is uniquely represented by $\hat{\zeta}_i$.

When $p=2$, let $\zeta_i = \eta_i^2$ for $i = 1, \dots, r$.
Let $S$ be the polynomial subring of $\Ext_{A}^{*}(k,k)$ generated  
by the $\zeta_i$'s, so that
\begin{equation}
\label{eq:ringS}
S = k[\zeta_1, \dots,\zeta_r]\,.
\end{equation}
The Bockstein map is an operation on cohomology that
raises degrees by one. If $p=2$ it coincides with the
Steenrod square. For each $i$, the Bockstein of the
cohomology class $\eta_i$ is the class $\zeta_i$. Thus,
when $k = \bfp$, the  subring $S$ is the subring generated
by the images of the degree one classes under the Bockstein map. 

Now consider the Hochschild cohomology.  
As before, set $A:= k[x]/(x^p)$. The enveloping
algebra $\env A$ is a truncated polynomial ring in
variables $y:= x \otimes 1$ and $z:= 1 \otimes x$, so
that  $\env A = k[y,z]/(y^p,z^p).$ The $\env A$ action
on $A$ is defined by the surjection  $\mu\colon \env A\to A$
that maps $y$ and $z$ to $x$.  Thus $\env A \cong A[y-z]/(y-z)^p.$
The kernel of $\mu$ is
the ideal $(y-z)$ and the minimal projective resolution
of $A$ as an $\env A$-module 
has the form:
\begin{equation}
\label{eq:Ares}
\xymatrix{
Q_{*}\colon & \dots \ar[r] & \env A \ar[r]^{y-z} & \env A \ar[r]^{\ (y-z)^{p-1}} 
& \env A \ar[r]^{y-z} &  \env A \ar[r] & 0
}
\end{equation}
with canonical augmentation $Q_{*}\to A$, also denoted $\mu$. 

Let $A:=k[x_{1},\dots,x_{r}]/(x_{1}^{p},\dots,x_{r}^{p})$ and
set  $A_i:=k[x_i]/(x_i^p)$. With $(Q^{(i)}_*, \mu_i)$ the
projective $\env {A_i}$-resolution of $A_{i}$ from \eqref{eq:Ares},
the complex
\begin{equation}
\label{eq:Ares2}
(Q_{*},\mu):= (Q^{(1)}_* \otimes_{k} \cdots \otimes_{k} Q_*^{(r)},
\mu_{1} \otimes \cdots \otimes \mu_r)\,.
\end{equation}
is a projective $\env A$-resolution of $A$.  Set 
\[
Q_{j_1, \cdots, j_r} := Q_{j_1}^{(1)} \otimes_{k} \cdots \otimes_{k}
Q_{j_r}^{(r)} = \env {A_1} \otimes_{k} \cdots \otimes_{k}
\env{A_r} \cong \env A\,.
\]
Let $\sigma_{j_1, \dots, j_r}\colon Q_* \to A$ be the map whose
restriction to $Q_{j_1, \dots, j_r}$ is the canonical map
$\env A\to A$ and whose restriction to $Q_{\ell_1, \dots, \ell_r}$
is zero if $j_i \ne \ell_i$ for some $i$. Let
$\hat{\delta_i}:= \sigma_{j_1, \dots, j_r}$ where $j_i = 1$ and
$j_\ell = 0$ for $\ell \neq i$ and
$\hat{\chi_i}:=\sigma_{j_1, \dots, j_r}$ where $j_i = 2$ and
$j_\ell = 0$ for $\ell \neq i$.
 
The Hochschild cohomology ring of $A$ over $k$ has the form
\[
\Ext_{\env A}^{*}(A,A) = \begin{cases} 
A[\delta_1, \dots, \delta_r] & \text{ if $p = 2$, } \\
\Lambda_{A}(\delta_1, \dots, \delta_r) \otimes_{A}
A[\chi_1, \dots, \chi_r]  &\text{ otherwise,}
\end{cases}
\]
with $\delta_{i}$ and $\chi_{i}$ the cohomology classes corresponding to $\hat{\delta_{i}}$ and $\hat{\chi_{i}}$ respectively.

\section{Changing the coalgebra structure on $kE$} 
\label{sec:eleab}
Let $k$ be a field of positive characteristic $p$
and  set $A=k[x_{1},\dots,x_{r}]/(x_{1}^{p},\dots,x_{r}^{p})$.
There are two often-used coalgebra structures on $A$
that  make it a Hopf algebra. 

The first comes from viewing $A$ as the group algebra
of an elementary abelian $p$-group, say
$\langle g_{1},\dots,g_{c}  \rangle$ with base field
of characteristic $p$. Then $A$ has comultiplication
$\gDelta\colon A \to A \otimes A$ given by
$g \mapsto g \otimes g$; equivalently, 
\[
\xymatrix{ 
\gDelta(x_{i}) =  x_{i} \otimes 1  +  x_{i} \otimes x_{i} + 1 
\otimes x_{i}\,.
}
\]
The antipode is the homomorphism of $k$-algebras
(note that $A$ is commutative) induced by the map
$g_{i}\mapsto g_{i}^{-1}$, which translates to
\[
\grS(x_{i})= (1+x_{i})^{-1} - 1 = -x_{i}+x_{i}^{2} - \cdots +x_{i}^{p-1}
\]

The other coalgebra structure on $A$ comes from viewing
it as the restricted enveloping algebra of the restricted
$p$-Lie algebra $k^{r}$, with trivial bracket and
$p$-power operation. Then the comultiplication
$\lDelta\colon A \to A \otimes A$ given by
\[
\xymatrix{ 
\lDelta(x_{i}) = x_{i} \otimes 1 + 1 \otimes x_{i}\,.
}
\]
The antipode is the homomorphism of $k$-algebras $A\to A$ defined by
\[
\LiS(x_{i}) = -x_{i}
\]
The different coalgebra structures induce different
actions of $\Ext^{*}_{A}(k,k)$ on the cohomology of modules;
see Example~\ref{ex:badhopf}. However, the actions do
agree on the subalgebra generated by the Bocksteins
of the degree one elements. This is the content of
Theorem~\ref{thm:twohopf}. A key step in its proof
is an explicit computation of the map $\Phi_{\Delta}$
from Construction~\ref{con:ktoA} for the different
coalgebra structures. In view of the computations
recalled in Section~\ref{sec:prelim}, this amounts
to describing the maps 
\[
\Phi_{\gDelta},\Phi_{\lDelta}\colon
k[\eta_{1},\dots,\eta_{r},\zeta_{1},\dots,\zeta_{r}]\lra
A[\delta_{1},\dots,\delta_{r},\chi_{1},\dots,\chi_{r}]
\]
from the cohomology of $A$ to its Hochschild cohomology. 

\begin{thm}  
\label{thm:phi-gr}
With the Hopf algebra structure on $A$ induced by
$\gDelta$ and $\grS$, the homomorphism
$\Phi_{\gDelta}\colon \Ext^{*}_{A}(k,k) \to \Ext^{*}_{\env A}(A,A)$
of $k$-algebras is given by
\[
\Phi_{\gDelta}(\eta_{i}) = (1+x_{i})\delta_{i} \quad
\text{and}\quad  \Phi_{\gDelta}(\zeta_{i}) = \chi_{i}\,
\quad\text{for $i=1,\dots,r$.}
\]
\end{thm}

\begin{proof}
  We first verify the result for $r=1$; that is to say,
  when $A = k[x]/(x^p)$.  In what follows we use the
  maps $\phi_{\gDelta}$ and $\iota$, and their properties,
  from Construction~\ref{con:ktoA} without comment.
  Let $P_{*}$ be the minimal projective resolution of
  $k$ over $A$ from \eqref{eq:kres} and $Q_{*}$ the
  minimal projective resolution of $A$ over $\env A$
  from \eqref{eq:Ares}. Applying $\phi_{\gDelta}$ to $P_{*}$
  yields a projective resolution of $A$ over $\env A$.
  This gives the top row in the following commutative
  diagram of complexes of $\env A$-modules: 
\[
\xymatrixcolsep{1.2pc}
\xymatrix{
  \cdots \ar[r] & \phi_{\gDelta}(A) \ar[rr]^{\phi_{\gDelta}(x)} &&
  \phi_{\gDelta}(A) \ar[rr]^{\phi_{\gDelta}(x^{p-1})}  
  && \phi_{\gDelta}(A) \ar[rr]^{\phi_{\gDelta}(x)} &&
  \phi_{\gDelta}(A) \ar@{->>}[rr]^{\phi_{\gDelta}(\varepsilon)} &&
  \phi_{\gDelta}(k)  \\
  \cdots \ar[r] & \env A \ar@{->}[u]^{\iota} \ar[rr]^{\frac{y-z}{1+z}} &&
  \env A  \ar@{->}[u]^{\iota} \ar[rr]^{(\frac{y-z}{1+z})^{p-1}}  
  && \env A \ar@{->}[u]^{\iota} \ar[rr]^{\frac{y-z}{1+z}} &&
  \env A \ar@{->}[u]^{\iota} \ar@{->>}[rr]^{\mu} &&
  A \ar@{->}[u]^{\cong} \\
  \cdots \ar[r] & \env A \ar@{->}[u]^{(1+z)} \ar[rr]_{y-z} &&
  \env A  \ar@{=}[u] \ar[rr]_{(y-z)^{p-1}}  
  && \env A \ar@{->}[u]^{(1+z)} \ar[rr]_{y-z} && \env A
  \ar@{=}[u] \ar@{->>}[rr]^{\mu} && A \ar@{=}[u]
}
\]
The bottom row is the augmentation of the minimal projective
resolution~\eqref{eq:Ares} of $A$ over $\env A$. It is
clear that the lower part of the diagram is commutative.
As to the upper part,  the commutativity of the square
on the top right corner is clear.
For the next square, we note that $\phi_{\gDelta}(x)$ is 
the map that takes $1 \otimes 1$ to $x \otimes 1$ in 
$\phi_{\gDelta}(A)$. However, this is not multiplication by 
the element $y = x \otimes 1$ in $\env A$. See Construction \ref{con:ktoA}.
Instead, we have
that 
\[
y(1 \otimes 1) = x \otimes 1  + x \otimes x + 1 \otimes x
\qquad \text{and} \qquad z(1 \otimes 1) = 1 \otimes x
\]
in $\phi_{\gDelta}(A)$. Hence, one has
\[
(y-z)(1 \otimes 1) = (1+z)(x \otimes 1)
\]
and $\phi_{\gDelta}(x)$ is multiplication by $(y-z)/(1+z)$ as asserted.  Likewise, $\phi_{\gDelta}(x^{p-1})$ is multiplication by $((y-z)/(1+z))^{p-1}$

It is clear from the construction that  the cocycle $\hat\eta$ and $\hat\zeta$, from $P_{*}\to k$, are mapped to the cocycles $(1+z)\hat\delta$ and  $\hat\chi$, respectively, from $Q_{*}\to A$. This yields the desired result. For later use we denote
\begin{equation}
\label{eq:lifts}
\kappa\colon Q_{*} \lra  \phi_{\gDelta}(P_{*})
\end{equation}
the morphism of complexes of $\env A$-modules constructed above.

Assume $r\ge 2$. Let $P_{*}$ be the resolution of $k$ over $A$, and let $Q_{*}$ be the resolution of $A$ over $\env A$. The tensor product, over $k$, of the morphisms $\kappa^{(i)}\colon Q^{(i)}_{*}\to \phi_{\gDelta}(P^{(i)}_{*})$ from \eqref{eq:lifts} yields a morphism 
\[
\kappa:= \kappa^{(1)}\otimes_{k}\cdots \otimes_{k}
\kappa^{(r)}\colon Q_{*}  \lra \phi_{\gDelta}(P_{*})\,,
\]
of complexes of $\env A$-modules that lifts the isomorphism
$A\cong \phi_{\gDelta}(k)$. Once again, it is evident, by inspection,
that  the cocycles $\hat\eta_{i}$ and $\hat\zeta_{i}$ are mapped
to the cocycles $(1+z)\hat\delta_{i}$ and  $\hat\chi_{i}$, respectively.
\end{proof}

An analogous argument gives also the next result.

\begin{thm}  
\label{thm:phi-lie}
With the Hopf algebra structure on $A$ induced by $\lDelta$
and $\LiS$, the homomorphism
$\Phi_{\lDelta}\colon \Ext^{*}_{A}(k,k) \to \Ext^{*}_{\env A}(A,A)$
of $k$-algebras is given by
\[
\Phi_{\lDelta}(\eta_{i}) =\delta_{i} \quad \text{and}\quad
\Phi_{\lDelta}(\zeta_{i}) = \chi_{i}\, \quad\text{for $i=1,\dots,r$}
\]
\end{thm}

\begin{proof}
  The key point, as in the proof of the preceding theorem, is to verify
  that one has a commutative diagram of complexes of $\env A$-modules: 
\[
\xymatrixcolsep{1.2pc}
\xymatrix{
  \cdots \ar[r] & \phi_{\lDelta}(A) \ar[rr]^{\phi_{\lDelta}(x)} &&
  \phi_{\lDelta}(A) \ar[rr]^{\phi_{\lDelta}(x^{p-1})}  
  && \phi_{\lDelta}(A) \ar[rr]^{\phi_{\lDelta}(x)} &&
  \phi_{\lDelta}(A) \ar@{->>}[rr]^{\phi_{\lDelta}(\varepsilon)} 
  && \phi_{\lDelta}(k)  \\
  \cdots \ar[r] & \env A \ar@{->}[u]^{\iota} \ar[rr]^{y-z} &&
  \env A  \ar@{->}[u]^{\iota} \ar[rr]^{(y-z)^{p-1}}  
  && \env A \ar@{->}[u]^{\iota} \ar[rr]^{y-z} && \env A
  \ar@{->}[u]^{\iota} \ar@{->>}[rr]^{\mu} &&  A \ar@{->}[u]^{\cong} \\
  \cdots \ar[r] & \env A \ar@{->}[u]^{(1+z)} \ar[rr]_{y-z} &&
  \env A  \ar@{=}[u] \ar[rr]_{(y-z)^{p-1}}  
  && \env A \ar@{->}[u]^{(1+z)} \ar[rr]_{y-z} && \env A
  \ar@{=}[u] \ar@{->>}[rr]^{\mu} && A \ar@{=}[u]
}
\]
This proof of the commutativity is similar to that of the previous case. 
\end{proof}

The next result is direct consequence of the preceding computations.

\begin{thm}  
\label{thm:twohopf}
Let $A=k[x_{1},\dots,x_{r}]/(x_{1}^{p},\dots,x_{r}^{p})$, with $k$
a field of positive characteristic $p$.
For any $A$-module $M$, the homomorphisms
\[
\Theta^M_{\gDelta},  \Theta^M_{\lDelta}\colon \Ext^{*}_{A}(k,k)
\lra  \Ext^*_A(M, M) 
\]
defined in \eqref{eq:ktoM} using the coalgebra maps $\gDelta$
and $\lDelta$, respectively, coincide on the subring
$S = k[\zeta_1, \dots, \zeta_r]$ of $\Ext^{*}_{A}(k,k)$
defined in \eqref{eq:ringS}. \qed
\end{thm}

\begin{rem}
  In ongoing work, in collaboration with Luchezar L. Avramov,
  we have been able to establish a version of the
  preceding theorem for more general finite dimensional
  commutative algebras; the techniques required are rather
  more involved and will be presented elsewhere. This
  raises that possibility that such a result may be true
  for any finite dimensional commutative Hopf algebra.
\end{rem}

\section{Tensor products of $L_{\zeta}$-modules} 
\label{sec:lzeta}
As in Section~\ref{sec:eleab}, let $k$ be a field of
positive characteristic $p$ and  set 
\[
A=k[x_{1},\dots,x_{r}]/(x_{1}^{p},\dots,x_{r}^{p})\,.
\]
Let $S$ be the subalgebra of $\Ext^*_{A}(k,k)$ identified
in \eqref{eq:ringS}. We investigate the circumstances
under which the tensor products of $L_\zeta$ modules
(see Definition~\ref{defi:Lzeta}) are independent of the
Hopf algebra structures on $A$ described in
Section~\ref{sec:eleab}. The main result is as follows;
see the paragraph preceding Corollary~\ref{cor:HH} for notation.

\begin{thm}  
\label{thm:tensorS}
Let $\zeta$ be a homogeneous element of $S$. For any
$A$-module $M$, there is an isomorphism
$\gDelta(L_{\zeta}\otimes_{k}M)\cong \lDelta(L_{\zeta}\otimes_{k}M)$
of $A$-modules.
\end{thm}

\begin{proof}
  The statement is a direct consequence of Corollary \ref{cor:HH}
  and Theorem \ref{thm:twohopf}. 
\end{proof}

\begin{cor}  
\label{cor:tensorLzeta} 
Suppose that $\zeta_1, \dots, \zeta_n$ are homogeneous elements
of positive degree in $\Ext^{*}_{A}(k,k)$. If all but one of
$\zeta_1, \dots, \zeta_n$ is in the subring $S$, then there
is an isomorphism of $A$-modules
\[
\gDelta(L_{\zeta_1} \otimes_{k} \cdots \otimes_{k} L_{\zeta_n}) \cong
\lDelta(L_{\zeta_1} \otimes_{k} \cdots \otimes_{k} L_{\zeta_n})\,.
\]
\end{cor}

\begin{proof}
  Without loss of generality, it may be assumed that
  $\zeta_1, \dots, \zeta_{n-1}$ are in $S$. The proof is
  by a backwards induction on $n$, the base case $n = 1$
  being a tautology. The induction hypothesis yields the
  second isomorphism below
\begin{align*}
\gDelta(L_{\zeta_1} \otimes_{k} \cdots \otimes_{k} L_{\zeta_n})  
&\cong \gDelta(L_{\zeta_1} \otimes_{k}
\gDelta(L_{\zeta_{2}}\otimes_{k} \cdots \otimes_{k} L_{\zeta_n})) \\
&\cong \gDelta(L_{\zeta_1} \otimes_{k}
\lDelta(L_{\zeta_{2}} \otimes_{k} \cdots \otimes_{k} L_{\zeta_n})) \\
&\cong \lDelta(L_{\zeta_1} \otimes_{k}
\lDelta(L_{\zeta_{2}} \otimes_{k} \cdots \otimes_{k} L_{\zeta_n})) \\
&\cong \lDelta(L_{\zeta_1} \otimes_{k} \cdots
\otimes_{k} L_{\zeta_n})  
\end{align*}
The third one is by Theorem~\ref{thm:tensorS}, and
the other two are standard.
\end{proof}

\begin{rem}
  The modules $L_\zeta$ have some remarkable properties.
  Under certain circumstances, the annihilator in
  $\HH^*(E,k)$ of $\Ext^*_{kE}(L_\zeta,L_\zeta)$ is the
  ideal generated by $\zeta$. This happens, for example,
  if $p > 2$ and $n$  is even \cite{Ca5}. In general,  the
  annihilator of the cohomology of $L_\zeta$ depends on the
  choice of  the coalgebra structure as we see in
  Example~\ref{ex:badhopf}. The sequence \eqref{eq:Lzeta}
  has a translation   
\[
\xymatrix{
  \CE_\zeta\colon \quad 0 \ar[r] & \Omega^1(k) \ar[r] &
  L_\zeta \oplus Q \ar[r]  & \Omega^{d}(k) \ar[r] & 0
}
\]
where $Q$ is the projective cover of the trivial module.
The translated sequence represents the cohomology class
$\zeta \in \Ext_{kE}^1(\Omega^{d}(k), \Omega^1(k)) \cong
\Ext^{d}_{kE}(k,k)$.  Consequently, $\zeta$ is in the
annihilator of $\Ext^*_{kE}(M,M)$ for a module $M$, if
and only if  $\CE_\zeta \otimes_{k} M$  splits. This is
equivalent to the requirement that there is a stable isomorphism
\[ 
L_{\zeta} \otimes_{k} M  \ \cong \ \Omega^{d}(M) \oplus \Omega^1(M)
\]
\end{rem}

The following example, noted already in \cite{Cvar, Con}, shows that the conclusion of Theorem~\ref{thm:tensorS} may fail
if $\zeta$ is not in $S$. 

\begin{exm} 
\label{ex:badhopf} 
Let $k$ be a field of characteristic $2$ and $E$ an elementary
abelian group of order~4; thus $\HH(E,k) = k[\eta_{1},\eta_{2}]$;
see Section~\ref{sec:prelim}.  Set
$\zeta = \eta_1 - \alpha\eta_2 \in \HH^1(E,k)$ where
$\alpha \in k$ with $\alpha \neq 0$ or $1$.  The module
$L_{\zeta}$ has a $k$-basis consisting of elements $u, v$
such that $x_1u = \alpha x_2u$ and $x_1v = x_2v$. Using
the Lie coalgebra structure, we can compute that
$L_{\zeta} \otimes_{k} L_{\zeta}$ is isomorphic to a direct
sum of two copies of $L_{\zeta}$ generated by $u \otimes u$
and $u \otimes v$.  However, with the group coalgebra
structure, $L_{\zeta} \otimes_{k} L_{\zeta}$ is indecomposable.
Indeed, under this structure, one has
\[
x_1(u \otimes u) = \alpha(u \otimes v) + \alpha(v \otimes u) + 
\alpha^2 (v \otimes v) = x_2(u \otimes u) + \alpha^2 (v \otimes v).
\]
and the last term that makes it impossible to decompose
$L_{\zeta}\otimes_{k}L_{\zeta}$.
\end{exm}

\begin{rem}
  Computer calculations using the computer algebra system
  Magma \cite{BCannon} give evidence that
  Corollary \ref{cor:tensorLzeta} might have a strong converse.
  In one experiment,  two random elements $\gamma_1$
  and $\gamma_2$ were chosen in  $\HH^4(E,k)$ with $E$
  an elementary abelian group of order 8 and $k$ the
  field with 8 elements.  The tensor product of modules
  $L_{\gamma_1}$ and $L_{\gamma_2}$ was taken using both of
  the coalgebra structures and the two results were
  compared. This operation was performed several times and
  in every  case, the two tensor products were isomorphic
  if and only if one of  the two chosen cohomology elements
  was in the subring $S$. The same experiment was performed
  taking two elements in degree two of an elementary
  abelian group of order 27 over a field of order 9, with
  the same result. 
\end{rem}

\section{An equality of varieties} 
\label{sec:remvar}
Let $E$ be an elementary abelian $p$-group and $k$ an
algebraically closed field of characteristic $p$. In the
paper \cite{AS}, Avrunin and Scott prove a conjecture of
the first author (see \cite{Cvar}) asserting the equivalence
of the support variety of a module with a rank variety for
that same module. For notation, let
$kE = k[x_1, \dots, x_r]/(x_1^p, \dots, x_r^p)$. Let $M$
be a $kE$-module. The support variety $V_G(M)$ of $M$ is
the closed subset of $\mathrm{Proj}\, \HH^*(G,k)$
consisting of all homogeneous prime ideals that contain
the annihilator $J(M)$ in $\HH^*(E,k)$ of the  cohomology
ring $\Ext^*_{kE}(M, M)$.  The rank variety of $M$,
denoted $V^r_G(M)$ is
the set of all points $[\alpha_1, \dots, \alpha_r]$
in $\bP^{r-1}$ such that $\alpha^*(M)$ is not a free module.
Here $\alpha\colon k[t]/(t^p) \to
kE$ is given by $\alpha(t) = \alpha_1x_1 + \dots + \alpha_rx_r$
and $\alpha^*(M)$ is the restriction of $M$ to a
$k[t]/(t^P)$-module along
the map $\alpha$.

The conjecture states that for $\alpha \in \bP^{r-1}$,
$\alpha \in V^r_G(M)$ if and only if $\alpha^*(J(M)) = \{0\}$,
where $\alpha^*(J(M))$ is the restriction of the ideal
to the cohomology ring of $k[t]/(t^p)$ along $\alpha$.
This all makes sense because $\alpha$ is a flat embedding.
The most difficult part is the proof of the assertion that
if $\alpha^*(J(M)) = \{0\}$, then $\alpha^*(M)$ is a free
module over $k[t]/(t^p)$. The proof by Avrunin and Scott
uses a spectral sequence argument under the assumption
that $kE$ has the coalgebra structure of the restricted
enveloping algebra of an elementary Lie algebra. In this
case any such $\alpha$ is a map of Hopf algebras and
this point is important in the proof.  The other key
step is their proof that the variety is independent of
the coalgebra structure.

This last step is an easy consequence of Theorem \ref{thm:twohopf}.
The point is to restrict to the subring $S$. The
annihilator in $S$ of
$\Ext^*_{kE}(M,M)$ is $S \cap J(M)$. Moreover
$V_G(J(M)) = V_G(S \cap J(M))$. While the ideal $J(M)$
depends on the coalgebra structure, the ideal $J(M)\cap S$
does not, by Theorem \ref{thm:twohopf}.

\end{document}